\documentclass[a4paper]{article}
\usepackage[utf8]{inputenc}
\usepackage{hyperref}
\usepackage{amsthm,amsmath,amssymb,amsfonts}
\usepackage{tikz}
\usepackage{palatino}

\hypersetup{
    colorlinks=true,
    linkcolor=black,
    urlcolor=blue,
    citecolor=black,
    pdftitle={Collapsing Constructive and Intuitionistic Modal Logics},
    pdfauthor={Leonardo Pacheco}
}

\theoremstyle{plain}
\newtheorem{theorem}{Theorem}
\newtheorem{lemma}[theorem]{Lemma}
\newtheorem{proposition}[theorem]{Proposition}
\newtheorem{corollary}[theorem]{Corollary}

\theoremstyle{definition}
\newtheorem{definition}[theorem]{Definition}
\newtheorem*{remark}{Remark}
\newtheorem*{problem}{Problem}

\newcommand{\tuple}[1]{{\langle #1 \rangle}}

\newcommand{\ck}{\mathsf{CK}}
\newcommand{\ik}{\mathsf{IK}}
\newcommand{\ckb}{\mathsf{CKB}}
\newcommand{\ikb}{\mathsf{IKB}}

\begin{document}

% Title
\title{Collapsing Constructive and Intuitionistic Modal Logics}

% Author
\author{Leonardo Pacheco \\
{\small Institute of Discrete Mathematics and Geometry, TU Wien, Austria.} \\
{\small \texttt{leonardovpacheco@gmail.com}}}
\date{}

\maketitle
%%%%%%%%%%%%%%%%%%%%%%%%%%%%%%%%%%%%%%%%%%%%%%%%%%%%%%%%%%%%%%%

\begin{abstract}
    We prove that the constructive and intuitionistic variants of the modal logic $\mathsf{KB}$ coincide.
    This result contrasts with a recent result by Das and Marin, who showed that the constructive and intuitionistic variants of $\mathsf{K}$ do not prove the same diamond-free formulas.
\end{abstract}

%%%%%%%%%%%%%%%%%%%%%%%%%%%%%%%%%%%%%%%%%%%%%%%%%%%%%%%%%%%%%%%
\section{Introduction}
%%%%%%%%%%%%%%%%%%%%%%%%%%%%%%%%%%%%%%%%%%%%%%%%%%%%%%%%%%%%%%%
Das and Marin \cite{das2023diamonds} recently (re)discovered that the constructive and intuitionistic variants of $\mathsf{K}$ do not prove the same diamond-free formulas.
We show that the constructive and intuitionistic variants of the modal logic $\mathsf{KB}$ coincide.

The modal logic $\ck$ was first studied by Mendler and de Paiva \cite{mendler2005constructive}, and $\ik$ was first studied by Fischer Servi \cite{fischerservi1977modal}.
Both logics consider non-interdefinable $\Box$ and $\Diamond$ modalities.
The main difference between these logics is the classically equivalent variants of the axiom $K$ they consider.
While $\ck$ only has $K_\Box := \Box(\varphi\to\psi) \to (\Box\varphi \to \Box \psi)$; and $K_\Diamond := \Box(\varphi\to\psi) \to (\Diamond\varphi \to \Diamond \psi)$; $\ik$ also includes the axioms $FS := (\Diamond \varphi \to \Box\psi) \to \Box(\varphi\to\psi)$; $DP := \Diamond (\varphi\lor\psi) \to \Diamond\varphi\lor\Diamond\psi$; and $N := \neg \Diamond \bot$.
For more information on $\ck$ and $\ik$, see \cite{das2023diamonds,simpson1994proof}.

The logic $\ikb$ is obtained by adding the axioms $B_\Box := P\to\Box\Diamond P$ and $B_\Diamond := \Diamond\Box P \to P$ to $\ik$. It was first studied by Simpson \cite{simpson1994proof}, who provided semantics and proved a completeness theorem for $\ikb$.
The logic $\ckb$ is similarly obtained from $\ck$. Its proof theory was studied by Arisaka, Das and Stra\ss burger \cite{arisaka2015nested}, who provided a complete nested sequent calculus for it.
As far as we are aware, there are no semantics for $\ckb$ in the literature.

While the axioms $FS$, $DP$, and $N$ describe the interaction between $\Box$ and $\Diamond$ modalities, they also allow $\ik$ to prove $\Diamond$-free formulas not provable in $\ck$.
This observation was already showed by Grefe's PhD thesis \cite{grefe1999fischer}, but not published.
A semantic characterization of the axioms $FS$, $DP$, and $N$ was recently given by de Groot, Shillito, and Clouston \cite{degroot2024semantical}.

The situation is quite different when the axiom $B$ is involved: $\ckb$ and $\ikb$ coincide.
Some of the works mentioned above already point to the collapse of $\ckb$ and $\ikb$.
Arisaka \emph{et al.} \cite{arisaka2015nested} already showed that $\ckb$ proves $DP$ and $N$.
The results of de Groot \emph{et al.} implies that the natural semantics for $\ckb$ validates $FS$, $DP$, and $N$.
We finish this line of argument and show that the two logics $\ckb$ and $\ikb$ prove the same formulas.

\paragraph{Outline}
In Section \ref{sec::preliminaries}, we define the logics studied in this paper and their respective semantics.
In Section \ref{sec::completeness}, we prove the completeness of $\ckb$ and $\ikb$ with respect to both $\ckb$-models and $\ikb$-models via a canonical model argument; key to our proof will be the fact that the canonical model for $\ckb$ is an $\ikb$-model.
In Section \ref{sec::conclusion}, we comment on a problem for future work.
\paragraph{Acknowledgements} This work was partially supported by FWF grant TAI-797.
%%%%%%%%%%%%%%%%%%%%%%%%%%%%%%%%%%%%%%%%%%%%%%%%%%%%%%%%%%%%%%%

%%%%%%%%%%%%%%%%%%%%%%%%%%%%%%%%%%%%%%%%%%%%%%%%%%%%%%%%%%%%%%%
\section{Preliminaries}
\label{sec::preliminaries}
%%%%%%%%%%%%%%%%%%%%%%%%%%%%%%%%%%%%%%%%%%%%%%%%%%%%%%%%%%%%%%%
\paragraph{Syntax}
Fix a set $\mathrm{Prop}$ of proposition symbols.
The \emph{modal formulas} are defined by the following grammar:
\[
    \varphi := P \mid \bot \mid \varphi\land\varphi  \mid \varphi\lor\varphi \mid \varphi\to\varphi \mid \Box\varphi \mid \Diamond\varphi.
\]
We define $\neg\varphi:= \varphi\to\bot$ and $\top:=\bot\to\bot$.

\begin{definition}
    
    A \emph{(modal) logic} is a set of formulas closed under necessitation and \emph{modus ponens}
    \[
        (\mathbf{Nec}) \; \frac{\varphi}{\Box\varphi} \;\;\;\;\;\text{ and }\;\;\;\;\;
        (\mathbf{MP}) \; \frac{\varphi \;\;\; \varphi\to\psi}{\psi}.
    \]
    $\ck$ is the least logic containing the all intuitionistic tautologies, $K_\Box := \Box(\varphi\to\psi) \to (\Box\varphi \to \Box \psi)$; and $K_\Diamond := \Box(\varphi\to\psi) \to (\Diamond\varphi \to \Diamond \psi)$.
\end{definition}

Before we define the other logics used in this paper, we define the following variations of the axiom $K$:
\begin{itemize}
    \item $FS := (\Diamond \varphi \to \Box\psi) \to \Box(\varphi\to\psi)$;
    \item $DP := \Diamond (\varphi\lor\psi) \to \Diamond\varphi\lor\Diamond\psi$; and
    \item $N := \neg \Diamond \bot$.
\end{itemize}
On the classical setting, the axioms $K_\Box$, $K_\Diamond$, $DP$, $FS$, $N$ are equivalent, but the same does not happen on the constructive setting.
We follow the notation of \cite{balbiani2021constructive}; these axioms are also respectively called $k_1$, $k_2$, $k_3$, $k_4$, $k_5$ by Simpson \cite{simpson1994proof} and $K_\Box$, $K_\Diamond$, $C_\Diamond$, $I_{\Diamond\Box}$, $N_\Box$ by Dalmonte \emph{et al.} \cite{dalmonte2020intuitionistic}.

We will also need the following dual versions of the modal axiom $B$:
\begin{itemize}
    \item $B_\Box := P\to\Box\Diamond P$; and
    \item $B_\Diamond := \Diamond\Box P \to P$.
\end{itemize}
Again, while $B_\Box$ and $B_\Diamond$ are equivalent in the classical setting, they are not equivalent in the constructive setting.

\begin{definition}
    Let $\Lambda$ be a logic and $X$ be a set of formulas.
    The logic $\Lambda + X$ is the least logic containing $\Lambda \cup X$.
    In particular, we will consider the following logics:
    \begin{itemize}
        \item $\ckb := \ck + \{B_\Box, B_\Diamond\}$;
        \item $\ik := \ck + \{FS, DP, N\}$; and
        \item $\ikb := \ik  + \{B_\Box, B_\Diamond\} = \ckb + \{FS, DP, N\}$.
    \end{itemize}
    For any logic $\Lambda \in \{\ck, \ik, \ckb, \ikb\}$ and formula $\varphi$, we write $\Lambda\vdash\varphi$ for $\varphi\in\Lambda$.
\end{definition}

It is already known that $\ckb$ proves two of the extra axioms in $\ikb$:
\begin{theorem}[Arisaka, Das, Stra\ss burger \cite{arisaka2015nested}]
    \label{thm::ads}
    $\ckb\vdash DP$ and $\ckb\vdash N$.
\end{theorem}
\begin{proof}
    Both are proved in \cite{arisaka2015nested} using nested sequents.
    We give a direct proof of $N$ in $\ckb$, since we will use it below.

    As $\bot\to\Box\bot$ is a tautology, $\ckb\vdash\Diamond\bot\to\Diamond\Box\bot$ by $\mathbf{Nec}$ and $K_\Diamond$.
    But $\Diamond\Box\bot\to\bot$ is an instance of $B_\Diamond$, so $\ckb\vdash\Diamond\bot\to\bot$ too.
\end{proof}

\paragraph{Semantics}
Our semantics is based on Mendler and de Paiva's $\ck$-models \cite{mendler2005constructive}.
Models for the logics $\ik$, $\ckb$, and $\ikb$ will be obtained by adding restrictions to $\ck$-models.
\begin{definition}
    A $\ck$-model is a tuple $M=\tuple{W, W^\bot ,\preceq, R, V}$ where: 
    \begin{itemize}
        \item $W$ is the set of \emph{possible worlds}; 
        \item $W^\bot\subseteq W$ is the set of \emph{fallible worlds};
        \item the \emph{intuitionistic relation} $\preceq$ is a reflexive and transitive relation over $W$;
        \item the \emph{modal relation} $R$ is a relation over $W$; and 
        \item $V:\mathrm{Prop}\to \mathcal{P}(W)$ is a \emph{valuation function}.
    \end{itemize}
    We require that, if $w \preceq v$ and $w\in V(P)$, then $v\in V(P)$; and that $W^\bot\subseteq V(P)$, for all $P\in\mathrm{Prop}$.
    We also require that, if $w\in W^\bot$ and either $w\preceq v$ or $w R v$, then $v\in W^\bot$.
\end{definition}
We follow Mendler and de Paiva's terminology \cite{mendler2005constructive} and call the worlds in $W^\bot$ by fallible worlds; these worlds are also known as \emph{sick} worlds \cite{veldman1976intuitionistic} and \emph{exploding} worlds \cite{ilik2010kripke}.

Fix a $\ck$-model $M = \tuple{W, W^\bot,\preceq, R,V}$.
Define the valuation of the formulas over $M$ by induction on the structure of the formulas:
\begin{itemize}
    \item $M,w\models P$ iff  $w\in V(P)$;
    \item $M,w\models \bot$ iff $w\in W^\bot$;
    \item $M,w\models \varphi\land\psi$ iff $M,w\models\varphi$ and $M,w\models\psi$;
    \item $M,w\models \varphi\lor\psi$ iff $M,w\models\varphi$ or $M,w\models\psi$;
    \item $M,w\models \varphi\to\psi$ iff, for all $v\in W$, if $w\preceq v$ and $M,v\models\varphi$, then $M,v\models\psi$;
    \item $M,w\models \Box\varphi$ iff, for all $v,u\in W$, if $w\preceq v$ and $vRu$, then $M,u\models\varphi$; and
    \item $M,w\models \Diamond\varphi$ iff, for all $v\in W$, there is $u$ such that $vRu$ and $M,u\models\varphi$.
\end{itemize}

Before we are ready to define models for the logics $\ik$, $\ckb$, and $\ikb$, we need the following definition:
\begin{definition}
    Let $M = \tuple{W, W^\bot,\preceq, R,V}$ be a $\ck$-model.
    \begin{itemize}
        \item The relation $R$ is forward confluent iff $wR  v$ and $w\preceq w'$ implies there is $v'$ such that $v\preceq w' R v'$. 
        \item The relation $R$ is backward confluent iff $wR  v \preceq v'$ implies there is $w'$ such that $w\preceq w' R v'$.

    \end{itemize}
    Forward confluence is also known by F1 \cite{simpson1994proof}, forth-up confluence \cite{aguilera2022-godelLTL}, and $\mathsf{C}_\Diamond$-strong \cite{degroot2024semantical}.
    Backward confluence is also known by F2 \cite{simpson1994proof}, back-up confluence \cite{aguilera2022-godelLTL}, and $\mathsf{I}_{\Diamond\Box}$-weak \cite{degroot2024semantical}.
    Forward and backward confluence are illustrated in Figure \ref{figure::confluences}.
\end{definition}
\begin{figure}[ht]
\centering
\tikzstyle{world}=[circle,draw,minimum size=5mm,inner sep=0pt]
\begin{tikzpicture}
    \node (w) at (0,0) {$w$};
    \node (w2) at (2,0) {$w'$};
    \node (v) at (0,-2) {$v$};
    \node (v2) at (2,-2) {$v'$};

    \draw[->] (w) -- (w2) node[midway,above] {$\preceq$};
    \draw[->] (w) -- (v) node[midway,left] {$R$};
    \draw[dashed,->] (v) -- (v2) node[midway,below] {$\preceq$};
    \draw[dashed,->] (w2) -- (v2) node[midway,right] {$R$};
\end{tikzpicture}
\begin{tikzpicture}
    \node (w) at (0,0) {$w$};
    \node (w2) at (2,0) {$w'$};
    \node (v) at (0,-2) {$v$};
    \node (v2) at (2,-2) {$v'$};

    \draw[dashed,->] (w) -- (w2) node[midway,above] {$\preceq$};
    \draw[->] (w) -- (v) node[midway,left] {$R$};
    \draw[->] (v) -- (v2) node[midway,below] {$\preceq$};
    \draw[dashed,->] (w2) -- (v2) node[midway,right] {$R$};
\end{tikzpicture}
\caption{Schematics for forward and backward confluence, respectively. Solid arrows correspond to the universal quantifiers and dashed arrows correspond to the existential quantifiers.}
\label{figure::confluences}
\end{figure}
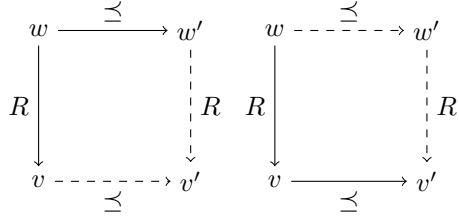

Proposition \ref{lem::classical-diamonds} below shows that, in a $\ck$-model where the modal relation is forward confluent, we can evaluate diamonds as in the classical setting.
In the rest of this paper, we will do so whenever possible (except in Lemma \ref{lem::soundness}, where we want make explicit the use of forward confluence).
\begin{proposition}
    \label{lem::classical-diamonds}
    Let $M = \tuple{W, W^\bot,\preceq, R,V}$ be a $\ck$-model where $R$ is a forward confluent relation and $\varphi$ be a modal formula.
    Then, for all $w\in W$, $M,w\models\Diamond\varphi$ iff there is $v\in W$ such that $wRv$ and $M,v\models\varphi$.
\end{proposition}
\begin{proof}
    Suppose $M,w\models\Diamond\varphi$. 
    As $\preceq$ is reflexive, $w\preceq w$. 
    So there is $v$ such that $wRv$ and $M,v\models\varphi$.
    Now, suppose there is $v$ such that $wRv$ and $M,v\models\varphi$.
    Let $w'$ be such that $w \preceq w'$.
    By forward confluence, there is $v'$ such that $w'Rv$ and $v\preceq v'$.
    By the monotonicity of $\preceq$, $M,v'\models\varphi$ too.
    Therefore $M,w\models\Diamond\varphi$.
\end{proof}

\begin{definition}
    Let $M = \tuple{W, W^\bot,\preceq, R,V}$ be a $\ck$-model.
    Then:
    \begin{itemize}
        \item $M$ is a \emph{$\ckb$-model} iff $R$ is symmetric, forward confluent, and backward confluent;
        \item $M$ is an \emph{$\ik$-model} iff $W^\bot =\emptyset$ and $R$ is forward and backward confluent; and
        \item $M$ is an \emph{$\ikb$-model} iff it is an $\ik$-model and $R$ is symmetric. Equivalently, $M$ is an $\ikb$-model iff it is an $\ckb$-model and $W^\bot = \emptyset$.
    \end{itemize}
    When the relation $R$ is symmetric, we denote it by $\sim$.

    For any logic $\Lambda \in \{\ck, \ik, \ckb, \ikb\}$ and formula $\varphi$, we write $\Lambda\models\varphi$ if $M,w\models\varphi$ for all $\Lambda$-model $M = \tuple{W, W^\bot,\preceq, R,V}$ and $w\in W$.
\end{definition}

Completeness for $\ck$, $\ik$, and $\ikb$ with respect to their respective classes of models is already known:
\begin{theorem}
    Let $\varphi$ be a formula.
    Then:
    \begin{itemize}
        \item $\ck\vdash\varphi$ iff $\ck\models\varphi$ \cite{mendler2005constructive};
        \item $\ik\vdash\varphi$ iff $\ik\models\varphi$ \cite{fischerservi1977modal};
        \item $\ikb\vdash\varphi$ iff $\ikb\models\varphi$ \cite{simpson1994proof}.
    \end{itemize}
\end{theorem}

It should be emphasized that, while $\ck$-models have no confluence requirement, $\ckb$-models do need both forward and backward confluence.
Forward confluence is necessary because $B_\Box$ is not sound over $\ck$-models whose modal relation is symmetric by not forward confluent.
Furthermore, in any $\ck$-model whose modal relation is symmetric, forward and backward confluence coincide. 
Therefore, we need the modal relation in a $\ckb$-model to be backward confluent too.
We show these necessities in the next two propositions.
\begin{proposition}
    \label{prop::confluence-is-necessary-kb}
    There is a $\ck$-model $M=\tuple{W, W^\bot, \preceq, \sim, V}$ and $w\in W$ such that $\sim$ is a symmetric relation and $B_\Box$ does not hold at $w$.
\end{proposition}
\begin{proof}
    Consider the model $M=\tuple{W, W^\bot, \preceq, \sim, V}$ where $W := \{w,v,v'\}$; $W^\bot:=\emptyset$; $\preceq := W\times W \cup \{\tuple{v,v'}\}$; $\sim := \{\tuple{w,v},\tuple{v,w}\}$; and $V(P) := \{w\}$. 
    This model is represented in Figure \ref{figure::confluence-is-necessary-kb}.
    The model $M$ satisfies all the requirements for $\ckb$-models but forward and backward confluence.
    Furthermore, $w\models P$ and $w\not\models \Box\Diamond P$; thus $B_\Box := P\to\Box\Diamond P$ does not hold at $w$.
    \begin{figure}[ht]
    \centering
    \tikzstyle{world}=[circle,draw,minimum size=5mm,inner sep=0pt]
    \begin{tikzpicture}
        \node (w)  at (0, 0) {$w$};
        \node (P)  at (0.6,0) {$\models P$};
        \node (v)  at (0, -1.5) {$v$};
        \node (v2) at (1.5, -1.5) {$v'$};

        \draw[<->] (w) -- (v) node[midway,left] {$\sim$};
        \draw[->] (v) -- (v2) node[midway,above] {$\preceq$};
    \end{tikzpicture}
    \caption{The $\ck$-model $M$ from Proposition \ref{prop::confluence-is-necessary-kb}, whose modal relation is symmetric but not forward confluent. $B_\Box$ fails at $w$.}
    \label{figure::confluence-is-necessary-kb}
    \end{figure}
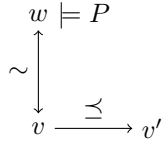
\end{proof}

\begin{proposition}
    \label{prop::confluences-are-equivalent}
    Let $M=\tuple{W, W^\bot, \preceq, \sim, V}$ be a $\ck$-model where $\sim$ is a symmetric relation over $W$.
    Then $\sim$ is forward confluent iff $\sim$ is backward confluent.
\end{proposition}
\begin{proof}
    Suppose $\sim$ is forward confluent.
    Let $v,v'$, be such that $w\sim v$ and $v\preceq v'$.
    As $\sim$ is a symmetric relation, $v\sim w$.
    By forward confluence, there is $w'$ such that $w\preceq w'$ and $v'\sim w'$.
    Therefore $w \preceq w' \sim v'$.
    The proof that backwards confluence implies forward confluence is similar.
    See also Figure \ref{figure::confluences} for a pictorial proof.
\end{proof}

Note that, by a recent result of de Groot, Shillito, and Clouston, $\ckb$-frames validate the axioms $DP$, $FS$, and $N$.
\begin{theorem}[De Groot, Shillito, Clouston \cite{degroot2024semantical}]
    \label{thm::dgsc}
    Let $M=\tuple{W, W^\bot, \preceq, R, V}$ be a $\ck$-model. Then:
    \begin{itemize}
        \item Suppose that, for all $w,v\in W$, $w R v$, and $v\in W^\bot$ implies $w\in W^\bot$. Then $M,w\models N$ for all $w\in W$.
        \item Suppose that $R$ is forward and backward confluent. Then $M,w\models DP$ and $M,w\models FS$ for all $w\in W$.
    \end{itemize}
\end{theorem}
\begin{corollary}
    $\ckb\models FS$, $\ckb\models DP$, and $\ckb\models N$.
\end{corollary}
\begin{proof}
    Let $M=\tuple{W, W^\bot, \preceq, \sim, V}$ be a $\ckb$-model.
    As $\sim$ is forward and backward confluent, $M,w\models DP$ and $M,w\models FS$, for all $w\in W$.
    Furthermore, if $w, v\in W$, $w \sim v$, and $v\in W^\bot$; then $v\sim w$ and so $w\in W^\bot$.
    Therefore $M,w\models N$, for all $w\in W$.
\end{proof}
%%%%%%%%%%%%%%%%%%%%%%%%%%%%%%%%%%%%%%%%%%%%%%%%%%%%%%%%%%%%%%%

%%%%%%%%%%%%%%%%%%%%%%%%%%%%%%%%%%%%%%%%%%%%%%%%%%%%%%%%%%%%%%%
\section{Completeness for \texorpdfstring{$\ckb$}{CKB} and \texorpdfstring{$\ikb$}{IKB}}
\label{sec::completeness}
%%%%%%%%%%%%%%%%%%%%%%%%%%%%%%%%%%%%%%%%%%%%%%%%%%%%%%%%%%%%%%%
In this section, we prove:
\begin{theorem}
    \label{thm::completeness-CKB-IKB}
    For all modal formula $\varphi$, the following are equivalent:
    \begin{enumerate}
        \item $\ckb\vdash\varphi$;
        \item $\ikb\vdash\varphi$;
        \item $\ckb\models\varphi$; and
        \item $\ikb\models\varphi$.
    \end{enumerate}
\end{theorem}
\begin{proof}   
    The implications $(1\Rightarrow 2)$ and $(3\Rightarrow 3)$ follow from the definitions: $\ckb$ is a subset of $\ikb$ and every $\ikb$-model is also a $\ckb$-model.
    The implications $(1\Rightarrow 3)$ and $(2\Rightarrow 4)$ follow from Lemma \ref{lem::soundness}.
    The implication $(4\Rightarrow 1)$ is the key implication of the proof; it follows from Lemma \ref{lem::completeness-CKB}.
    We outline the proof of this theorem in Figure \ref{fig::outline-completeness}.
\end{proof}
\noindent Note that the equivalence between items $(2)$ and $(4)$ was already proved by Simpson \cite{simpson1994proof}.

\begin{figure}
\centering
\begin{tikzpicture}
    \node (i-proof) at (0,0) {$\ikb\vdash\varphi$};
    \node (c-proof) at (5,0) {$\ckb\vdash\varphi$};
    \node (i-model) at (0,-2) {$\ikb\models\varphi$};
    \node (c-model) at (5,-2) {$\ckb\models\varphi$};

    \draw[->] (i-proof) -- (i-model) node[midway,left] {Lemma \ref{lem::soundness}};
    \draw[->] (c-proof) -- (c-model) node[midway,right] {Lemma \ref{lem::soundness}};
    \draw[->] (c-proof) -- (i-proof) node[midway,above] {By definition.};
    \draw[->] (c-model) -- (i-model) node[midway,below] {By definition.};
    \draw[->] (i-model) -- (c-proof) node[midway,sloped, above] {Lemma \ref{lem::completeness-CKB}};
\end{tikzpicture}
\caption{The outline of the proof of Theorem \ref{thm::completeness-CKB-IKB}.}
\label{fig::outline-completeness}
\end{figure}
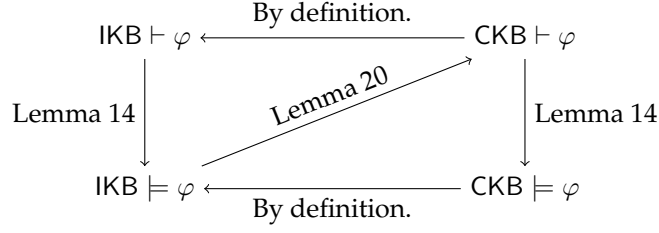

\subsection{Soundness}
We first prove that $\ckb$ is sound with respect to $\ckb$-models and that $\ikb$ is sound with respect to $\ikb$-models.
\begin{lemma}
    \label{lem::soundness}
    Let $\varphi$ be a formula. If $\ckb\vdash\varphi$, then $\ckb\models\varphi$; and, if $\ikb\vdash\varphi$, then $\ikb\models\varphi$.
\end{lemma}
\begin{proof}
    The proof that $\ckb$-models satisfy $K_\Box$ and $K_\Diamond$ and $\ikb$-models further satisfies $DP$, $FS$, and $N$ uses standard arguments.
    See, for example, Simpson's PhD thesis \cite{simpson1994proof} for detailed proofs.
    We only show that $\ckb$-models (and so $\ikb$-models) satisfy $B_\Box$ and $B_\Diamond$.

    Fix a $\ckb$-model $M = \tuple{W, W^\bot, \preceq, \sim, V}$.
    Suppose $w\in W$ and $w\models\Diamond\Box P$.
    Therefore, there is $v$ such that $w\sim v$ and $v\models\Box\varphi$.
    As $\sim$ is symmetric, $v\sim w$ and so $w\models P$.
    
    Suppose that $w\in W$ and $w\models P$.
    We want to show that $M,w\models \Box\Diamond P$.
    That is, we want to show that, for all $v, v', u$ such that $w\preceq v\sim v' \preceq u$, there is $u'$ such that $u\sim u'$ and $M,u'\models P$.
    To see such $u'$ exists, let $v, v', u$ be such that $w\preceq v\sim v' \preceq u$.
    By forward confluence, there is $u'$ such that $v\preceq u'$ and $u\sim u'$.
    Since $w\preceq v \preceq u'$, we have $M,u'\models P$ as we wanted.
\end{proof}

\subsection{Canonical model for \texorpdfstring{$\ckb$}{CKB}}
A \emph{(consistent) $\ckb$-theory} $\Gamma$ is a set of formulas such that:
    $\Gamma$ contains all the axioms of $\ckb$;
    $\Gamma$ is closed under $\mathbf{MP}$;
    if $\varphi\lor\psi\in\Gamma$, then $\varphi\in\Gamma$ or $\psi\in\Gamma$; and
    $\bot\not\in\Gamma$.
If $X$ is a set of formulas, then define $X^\Box := \{\varphi \mid \Box\varphi\in X\}$ and $X^\Diamond := \{\varphi \mid \Diamond\varphi\in X\}$.

Define the $\ckb$-canonical model $M_c := \tuple{W_c, W^\bot_c, \preceq_c,\sim_c, V_c}$ by:
\begin{itemize}
    \item $W_c := \{ \Gamma \mid \Gamma\text{ is a $\ckb$-theory} \}$;
    \item $W^\bot_c = \emptyset$;
    \item $\Gamma \preceq_c \Delta$ iff $\Gamma\subseteq\Delta$;
    \item $\Gamma \sim_c \Delta$ iff $\Gamma^\Box \subseteq \Delta$ and $\Delta\subseteq \Gamma^\Diamond$; and
    \item $\Gamma\in V_c(\varphi)$ iff $P\in\Gamma$.
\end{itemize}

\begin{remark}
    The definition of the $\ckb$-canonical model $M_c$ is standard for \emph{intuitionistic} modal logics.
    In general, using only theories for worlds is not sufficient for constructive modal logics; it works here because $\ckb$ and $\ikb$ coincide.
    One alternative is to pair theories with a set formulas which cannot be satisfiable in accessible worlds, as in \cite{alechina2001categorical,balbiani2021constructive}.
    Another alternative is to pair theories with the set of theories they can access, as in \cite{wijesekera1990constructive,degroot2024semantical}.
\end{remark}

\begin{lemma}
    \label{lem::sim_is_equiv}
    The relation $\sim_c$ is a symmetric relation.
\end{lemma}
\begin{proof}
    Suppose $\Gamma$ and $\Delta$ are $\ckb$-theories such that $\Gamma \sim_c \Delta$.
    Let $\Box\varphi\in\Delta$, then $\Diamond\Box\varphi\in\Gamma$ as $\Delta\subseteq \Gamma^\Diamond$.
    By $B_\Diamond$ and $\mathbf{MP}$, $\varphi\in \Gamma$.
    Therefore $\Delta^\Box\subseteq \Gamma$.
    Now, let $\varphi\in\Gamma$.
    Then $\Box\Diamond\varphi\in\Gamma$ by $B_\Diamond$ and $\mathbf{MP}$.
    Thus $\Diamond \varphi\in\Delta$, as $\Gamma^\Box\subseteq\Delta$.
    That is, $\Gamma\subseteq\Delta^\Diamond$.
    As $\Delta^\Box\subseteq \Gamma$ and $\Gamma\subseteq\Delta^\Diamond$, we conclude that $\Delta\sim_c\Gamma$.
\end{proof}

\begin{lemma}
    \label{lem::sim_is_confluent}
    The relation $\sim_c$ is backward confluent.
    That is, if $\Gamma,\Delta,\Sigma\in W_c$ and $\Gamma\sim_c\Delta\preceq_c\Sigma$, then there is $\Theta\in W_c$ such that $\Gamma\preceq_c\Theta\sim_c\Sigma$.
\end{lemma}
\begin{proof}
    Suppose $\Gamma\sim_c\Delta\preceq_c\Sigma$.
    By definition, $\Gamma^\Box\subseteq\Delta$, $\Delta\subseteq \Gamma^\Diamond$, and $\Delta\subseteq \Sigma$.
    Let $\Upsilon$ be the closure of $\Gamma \cup \{\Diamond\varphi \mid \varphi\in\Sigma\}$ under $\mathbf{MP}$.
    
    We first show that, if $\Box\varphi$ is a provable formula in $\Upsilon$, then $\varphi\in\Sigma$.
    There are formulas $\psi\in \Gamma$ and $\chi_0,\dots, \chi_n\in \Sigma$ such that
    \[
        \ckb \vdash (\bigwedge_{j<n}\Diamond\chi_j) \land \psi \to \Box \varphi.
    \]
    By $\mathbf{Nec}$ and $K$,
    \[
        \ckb \vdash (\bigwedge_{j<n}\Box\Diamond\chi_j) \to \Box( \psi \to \Box \varphi)
    \] and so
    \[
        \ckb\vdash (\bigwedge_{j<n}\Box\Diamond\chi_j) \to (\Diamond\psi \to \Diamond\Box \varphi).
    \]
    Since each $\chi_j$ is in $\Sigma$, so are the formulas $\Box\Diamond\chi_j$, by $B_\Box$ along with $\mathbf{MP}$.
    Since $\psi\in\Gamma$, $\Diamond\psi\in\Delta$, and thus $\Diamond\psi\in\Sigma$ too.
    By repeated applications of $\mathbf{MP}$, we have $\Diamond\Box\varphi\in\Sigma$.
    By $B_\Diamond$, we have $\varphi\in\Sigma$.

    Furthermore, $\bot\not\in\Upsilon$. To see that, suppose otherwise, then $\Box\bot\in\Upsilon$ by $\mathbf{MP}$, and so $\bot\in\Sigma$, which is impossible as $\Sigma$ is a $\ckb$-theory.

    Therefore $\Upsilon$ is a set such that: $\Gamma\subseteq\Upsilon$, $\Upsilon^\Box\subseteq\Sigma$, $\Sigma\subseteq\Upsilon^\Diamond$, and $\bot\not\in\Upsilon$.
    $\Upsilon$ might not be a theory, but we can extend it to a theory using Zorn's Lemma.
    
    To see so, let $\Theta$ be the maximal (with respect to the subset relation) set of formulas such that: $\Gamma\subseteq\Upsilon$, $\Upsilon^\Box\subseteq\Sigma$, $\Sigma\subseteq\Upsilon^\Diamond$, and $\bot\not\in\Upsilon$.
    Suppose $\varphi\lor\psi\in\Theta$.
    Then if $\varphi\not\in\Theta$ and $\psi\not\in\Theta$, we would have that $\neg\varphi\in\Theta$ and $\neg\psi\in\Theta$. 
    By $\mathbf{MP}$, we would have $\neg(\varphi\lor\psi)\in \Theta$, a contradiction.
    Therefore $\Theta$ is a $\ckb$-theory.
    By construction, we have that $\Gamma\preceq_c\Theta\sim_c\Sigma$.
\end{proof}

\begin{lemma}
    \label{lem::mc-is-canonical}
    The $\ckb$-canonical model $M_c$ is an $\ikb$-model.
\end{lemma}
\begin{proof}
    By definition, $M_c$ is an $\ck$-model with $W^\bot = \emptyset$.
    By Lemmas \ref{lem::sim_is_equiv} and \ref{lem::sim_is_confluent}, $\sim$ is symmetric and forward confluent. By Proposition \ref{prop::confluences-are-equivalent}, $\sim$ is backward confluent too.
\end{proof}

In the next lemma, we isolate the trickier application of Zorn's Lemma in the Truth Lemma:
\begin{lemma}
    \label{lem::verbose-proof-for-box}
    Let $\varphi$ be a formula and $\Gamma$ be a $\ckb$-theory.
    Then $\Box\varphi\not\in\Gamma$ implies that there are $\ckb$-theories $\Delta$ and $\Sigma$ such that $\Gamma\preceq_c\Delta\sim_c\Sigma$ and $\varphi\not\in\Sigma$.
\end{lemma}
\begin{proof}
    Suppose $\Box\varphi\not\in\Gamma$.
    Let $\Upsilon$ be the closure under $\mathbf{MP}$ of $\Gamma^\Box$ and $\Phi$ be the closure under $\mathbf{MP}$ of $\Gamma \cup \{\Diamond\varphi \mid \psi\in\Upsilon\}$.
    By the choice of $\Upsilon$, $\varphi\not\in\Upsilon$ and $\bot\not\in\Upsilon$, otherwise $\Box\varphi\in\Gamma$.
    We also have that $\bot\not\in\Phi$.
    Suppose otherwise, then there are formulas $\chi\in\Gamma$ and $\psi_0,\dots,\psi_n\in\Upsilon$ such that
    \[
        \ckb\vdash (\bigwedge_{i<n}\Diamond\psi_i) \to (\chi\to \bot).
    \]
    By $\mathbf{Nec}$ along applications of $K_\Box$ and $K_\Diamond$ with $\mathbf{MP}$, we have
    \[
        \ckb\vdash (\bigwedge_{i<n}\Box\Diamond\psi_i) \to (\Diamond\chi\to \Diamond\bot).
    \]
    Since each $\psi_i$ is in $\Upsilon$, so are the $\Box\Diamond\psi_i$ in $\Upsilon$.
    Since $\chi\in\Gamma$, $\Box\Diamond\chi\in\Gamma$ too, and so $\Diamond\chi\in\Upsilon$.
    Therefore $\Diamond\bot \in \Upsilon$, and so $\bot\in\Upsilon$ too since $\ckb\vdash N$; this is a contradiction.

    Consider now the pairs of sets of formulas $\tuple{X,Y}$ such that $\Upsilon\subseteq X$, $\Phi\subseteq Y$, $\Gamma\subseteq Y$, $Y^\Box\subseteq X$, $Y\subseteq X^\Diamond$, $\varphi\not\in X\cup Y$, and $\bot\not\in X\cup Y$.
    Consider the ordering $\leq$ where $\tuple{X,Y}\leq \tuple{X',Y'}$ iff $X\subseteq X'$ and $Y\subseteq Y'$.
    Note that $\tuple{\Upsilon,\Phi}$ is a pair which satisfies these properties.
    So, by Zorn's Lemma, there is a pair $\tuple{\Sigma,\Delta}$ which is maximal with respect to $\leq$.
    As in the proof of Lemma \ref{lem::sim_is_confluent}, if $\varphi\lor\psi\in \Sigma$ then $\varphi\in\Sigma$ or $\psi\in\Sigma$; and if $\varphi\lor\psi\in \Delta$ then $\varphi\in\Delta$ or $\psi\in\Delta$.
    Therefore $\Sigma$ and $\Delta$ must be $\ckb$-theories.
    Furthermore, $\Gamma\preceq_c\Delta\sim_c\Sigma$ and $\varphi\not\in\Sigma$ as we wanted.
\end{proof}

\begin{lemma}
    \label{lem::truth-lemma-CKB}
    Let $M_c$ be the $\ckb$-canonical model.
    For all formula $\varphi$ and for all $\ckb$-theory $\Gamma$,
    \[
        M_c,\Gamma\models \varphi \text{ iff } \varphi\in\Gamma.
    \]
\end{lemma}
\begin{proof}
    The proof is by structural induction on modal formulas.
    \begin{itemize}
        \item If $\varphi = P$, then, for all $\Gamma\in W_c$, we have $M_c,\Gamma\models P$ iff $\Gamma\in V_c(P)$ iff $P\in\Gamma$.

        \item If $\varphi = \bot$, then the lemma holds because $W_c^\bot = \emptyset$ and there is no $\Gamma\in W_c$ such that $\bot\in\Gamma$.

        \item If $\varphi = \psi_1\land\psi_2$, then, for all $\Gamma\in W_c$,
        \begin{align*}
            M_c,\Gamma\models \psi_1\land\psi_2 &\text{ iff } M_c,\Gamma\models\psi_1 \text{ and } M_c,\Gamma \models \psi_2 \\
            &\text{ iff } \psi_1\in\Gamma \text{ and }\psi_2\in\Gamma \\
            &\text{ iff } \psi_1\land \psi_2\in\Gamma.
        \end{align*}

        \item If $\varphi = \psi_1\lor\psi_2$, then, for all $\Gamma\in W_c$,
        \begin{align*}
            M_c,\Gamma\models \psi_1\lor\psi_2 &\text{ iff } M_c,\Gamma\models\psi_1 \text{ or } M_c,\Gamma \models \psi_2 \\
            &\text{ iff } \psi_1\in\Gamma \text{ or }\psi_2\in\Gamma \\
            &\text{ iff } \psi_1\lor \psi_2\in\Gamma.
        \end{align*}
        Here we use that if $\psi_1\lor\psi_2\in\Gamma$ then $\psi_1\in\Gamma$ or $\psi_2\in\Gamma$, as $\Gamma$ is a $\ckb$ theory.

        \item Let $\varphi := \psi_1\to\psi_2$ and $\Gamma\in W_c$.
        First suppose that $\psi_1\to\psi_2\in\Gamma$.
        Let $\Delta$ be a theory such that $\Gamma\preceq_c\Delta$ and $M_c,\Delta\models\psi_1$.
        By the induction hypothesis, $\psi_1\in\Delta$.
        As $\Gamma\preceq_c\Delta$, $\psi_1\to\psi_2\in\Delta$.
        By $\mathbf{MP}$, $\psi_2\in\Delta$.
        So $\Gamma\models\psi_1\to\psi_2$.

        Now suppose that $\psi_1\to\psi_2\not\in\Gamma$.
        Take $\Upsilon$ to be the closure of $\Gamma\cup \{\psi_1\}$ under $\mathbf{MP}$.
        If $\psi_2\in\Upsilon$, then there is $\chi\in \Gamma$ such that $\ckb\vdash(\chi \land \psi_1)\to \psi_2$.
        And so $\ckb\vdash\chi \to (\psi_1\to \psi_2)$.
        As $\chi\in\Gamma$, this means $\psi_1\to\psi_2\in \Gamma$, a contradiction.
        Therefore $\psi_2\not\in\Upsilon$.
        By Zorn's Lemma, we can build a theory $\Sigma$ such that $\Upsilon\subseteq \Sigma$ and $\psi_2\not\in\Sigma$.
        By the induction hypothesis, $M_c,\Sigma\models \psi_1$ and $M_c,\Sigma\not\models\psi_2$.
        As $\Gamma\preceq_c\Sigma$, we have that $M_c,\Gamma\not\models\psi_1\to\psi_2$.

        \item  Let $\varphi = \Box\psi$ and $\Gamma\in W_c$.
        First suppose that $\Box\psi\in\Gamma$.
        Let $\Gamma\preceq_c\Delta\sim_c\Sigma$.
        Then $\Box\psi\in\Delta$ and $\psi\in\Sigma$.
        By induction hypothesis, $M_c,\Sigma\models\psi$.
        So $M_c,\Gamma\models\Box\psi$.

        Now suppose that $\Box\psi\not\in\Gamma$. 
        By Lemma \ref{lem::verbose-proof-for-box}, there are $\ckb$-theories $\Delta$ and $\Sigma$ such that $\Gamma\preceq_c\Delta\sim_c\Sigma$ and $\varphi\not\in\Sigma$.
        By the induction hypothesis, $M_c,\Sigma\not\models\varphi$.
        Therefore $M_c,\Gamma\not\models\Box\varphi$.

        \item Let $\varphi = \Diamond\psi$ and $\Gamma\in W_c$.
        First suppose that $\Diamond\psi\in\Gamma$.
        Let $\Upsilon$ be the closure under $\mathbf{MP}$ of $\Gamma^\Box\cup\{\psi\}$.
        $\Gamma^\Box\subseteq\Upsilon$ holds by definition.
        Let $\theta\in\Upsilon$, then $\chi\land \psi \to \theta\in \ckb$ for some $\chi\in\Gamma^\Box$.
        Thus $\chi\to (\psi\to\theta)\in \ckb$ and $\Box\chi\to \Box(\psi\to\theta)\in \ckb$ by $K$.
        So $\Box(\psi\to\theta)\in \Gamma$.
        By $K$, $\Diamond\psi\to\Diamond\theta\in\Gamma$.
        So $\Diamond\theta\in\Gamma$.
        By Zorn's Lemma, there is a theory $\Delta$ such that $\Gamma\sim_c\Delta$ and $\psi\in\Delta$.
        By induction hypothesis, $\Delta\models\psi$.
        Therefore $\Gamma\models\Diamond\psi$.

        Now suppose that $\Diamond\psi\not\in\Gamma$.
        Let $\Delta$ be such that $\Gamma\sim_c\Delta$ and $\psi\in\Delta$.
        By the definition of $\sim_c$, $\Delta\subseteq \Gamma^\Diamond$, so $\psi\in\Gamma^\Diamond$.
        Therefore $\Diamond\psi\in\Gamma$, a contradiction.
        We conclude that, for all $\Delta\in W_c$, if $\Gamma\sim_c\Delta$, then $\psi\not\in\Delta$.
        By the induction hypothesis, $M_c,\Delta\not\models\psi$.
        Therefore $M_c,\Gamma\not\models\Diamond\psi$.
    \end{itemize}
    This finishes the proof of Lemma \ref{lem::truth-lemma-CKB}.
\end{proof}

\begin{lemma}
    \label{lem::completeness-CKB}
    Let $\varphi$ be a modal formula. If $\ikb\models\varphi$ then $\ckb\vdash\varphi$.
\end{lemma}
\begin{proof}
    Suppose $\ckb\not\vdash\varphi$.
    By Zorn's Lemma, there is a $\ckb$-theory $\Gamma$ not containing $\varphi$.
    By the Truth Lemma, $M_c,\Gamma\not\models\varphi$.
    Therefore $\ikb\not\models\varphi$.
\end{proof}
%%%%%%%%%%%%%%%%%%%%%%%%%%%%%%%%%%%%%%%%%%%%%%%%%%%%%%%%%%%%%%%

%%%%%%%%%%%%%%%%%%%%%%%%%%%%%%%%%%%%%%%%%%%%%%%%%%%%%%%%%%%%%%%
\section{Conclusion}
\label{sec::conclusion}
%%%%%%%%%%%%%%%%%%%%%%%%%%%%%%%%%%%%%%%%%%%%%%%%%%%%%%%%%%%%%%%
We showed that that the constructive and intuitionistic variations of $\mathsf{KB}$ coincide.
This is in contrast to the constructive and intuitionistic variations of $\mathsf{K}$, which do not prove the same $\Diamond$-free formulas.
This also implies that constructive and intuitionistic variations of $\mathsf{DB}$, $\mathsf{TB}$, $\mathsf{KB5}$, and $\mathsf{S5}$ coincide.
See \cite{arisaka2015nested} and \cite{simpson1994proof} for definitions of these logics.

We close the paper with a problem.
A \emph{$\ck$-frame} is a $\ck$-model without a valuation function. That is, a $\ck$-frame is a tuple $F=\tuple{W, W^\bot ,\preceq, R}$ where: $W$ is the set of \emph{possible worlds};  $W^\bot\subseteq W$ is the set of \emph{fallible worlds}; the \emph{intuitionistic relation} $\preceq$ is a reflexive and transitive relation over $W$; and the \emph{modal relation} $R$ is a relation over $W$.
A $\ck$-frame $F=\tuple{W, W^\bot ,\preceq, R}$ \emph{validates} a formula $\varphi$ iff, for all $\ck$-model obtained by adding a valuation to $F$ and all world $w\in W$, $M,w\models \varphi$.
\begin{problem}
    Characterize necessary and sufficient conditions for $\ck$-frames to validate the following axioms:
    \begin{itemize}
        \item $B_\Box := P \to \Box\Diamond P$, $B_\Diamond := \Diamond\Box P \to P$;
        \item $4_\Box := \Box\Box P \to \Box P$, $4_\Diamond := \Diamond\Diamond P \to \Diamond P$;
        \item $5_\Box := \Diamond P \to \Box\Diamond P$, $5_\Diamond := \Diamond\Box P \to \Box P$;
        \item $T_\Box := \Box P \to P$, $T_\Diamond := P \to \Diamond P$; and
        \item $D := \Box P \to \Diamond P$.
    \end{itemize}
\end{problem}

%%%%%%%%%%%%%%%%%%%%%%%%%%%%%%%%%%%%%%%%%%%%%%%%%%%%%%%%%%%%%%%
\bibliographystyle{alphaurl}
\bibliography{bibliography}

\end{document}